\documentclass[11pt,fleqn]{amsart}
\usepackage{amsmath, amsthm, amssymb, amsfonts, verbatim,color}
\usepackage[pdftex]{graphicx}
\usepackage{csquotes}
\usepackage[bookmarks]{hyperref}
\reversemarginpar            

\definecolor{darkgreen}{rgb}{0,0.55,0}

\newtheorem{theorem}{Theorem}[section]

\newtheorem{lemma}[theorem]{Lemma}
\newtheorem{proposition}[theorem]{Proposition}
\theoremstyle{definition}

\theoremstyle{remark}
\newtheorem{remark}[theorem]{Remark}

\numberwithin{equation}{section}
\numberwithin{theorem}{section}

\newcommand{\norm}[1]{\left\Vert#1\right\Vert}

\newcommand{\abs}[1]{\left\vert#1\right\vert}
\newcommand{\R}{{\mathbb{R}}}

\DeclareMathOperator*{\esslim}{ess\,lim}

\def\be{\begin{equation}}
\def\ee{\end{equation}}

\def\({\left(}
\def\){\right)}

\def\R{\mathbb{R}}

\def\e{\varepsilon}

\begin{document}
\title[Stability of shocks for finite-entropy solutions]{On the $L^2$ stability of shock waves for finite-entropy solutions of Burgers}

\date{\today}

\author[]{Andres A. Contreras Hip}

\author[]{Xavier Lamy}

\address[A.~Contreras Hip]{Department of Mathematical Sciences, New Mexico State University, Las Cruces,
New Mexico, USA}
\email{albertch@nmsu.edu} 

\address[X.~Lamy]{Institut de Math\'ematiques de Toulouse; UMR 5219, Universit\'e de Toulouse; CNRS, UPS IMT, F-31062 Toulouse Cedex 9, France}
\email{xlamy@math.univ-toulouse.fr}

\begin{abstract}
We prove $L^2$ stability estimates for entropic shocks among weak, possibly \emph{non-entropic}, solutions of scalar conservation laws $\partial_t u+\partial_x f(u)=0$ with strictly convex flux function $f$. This generalizes previous results by Leger and Vasseur, who proved $L^2$ stability among entropy solutions. Our main result, the estimate
\begin{align*}
\int_{\mathbb R} |u(t,\cdot)-u_0^{shock}(\cdot -x(t))|^2\,dx\leq \int_{\mathbb R}|u_0-u_0^{shock}|^2 +C\mu_+([0,t]\times\R),
\end{align*}
for some Lipschitz shift $x(t)$, includes an error term accounting for the positive part of the entropy production measure $\mu=\partial_t(u^2/2)+\partial_x q(u)$, where $q'(u)=uf'(u)$.
Stability estimates in this general non-entropic setting are of interest in connection with large deviation principles for the hydrodynamic limit of asymmetric interacting particle systems.
 Our proof adapts the scheme devised by Leger and Vasseur, where one constructs a shift $x(t)$ which allows to bound from above the time-derivative of the left-hand side. The main difference lies in the fact that our solution $u(t,\cdot)$ may present a non-entropic shock at $x=x(t)$ and new bounds are needed in that situation. We also generalize this stability estimate to initial data with bounded variation.
\end{abstract}

\maketitle

\section{Introduction}

We consider bounded weak (not necessarily entropy) solutions of Burgers' equation
\begin{align*}
\partial_t u +\partial_x \frac{u^2}{2} =0,
\end{align*}
or more generally a scalar conservation law
\begin{align}\label{eq:scl}
\partial_t u +\partial_x f(u) =0,\quad t\geq 0,\; x\in\R,
\end{align}
with uniformly convex flux $f''\geq \alpha >0$. Let us recall that for any entropy-flux pair $(\eta,q)$ i.e. $\eta''\geq 0$ and $q'=\eta'f'$, the corresponding entropy production of a bounded weak solution $u$ is the distribution
\begin{align}\label{eq:mueta}
\mu_\eta =\partial_t\eta(u) +\partial_x q(u).
\end{align}
In the special case $\eta(t)=t^2/2$ we will drop the subscript $\eta$ and simply write
\begin{align}\label{eq:mu}
\mu =\partial_t \frac{u^2}{2} + \partial_x q(u),\qquad q(v)=\int_0^v tf'(t)\, dt.
\end{align}
For smooth solutions the entropy production $\mu_\eta$ is always zero, but smooth long-time solutions do not exist in general. Entropy solutions are weak solutions whose entropy production is nonpositive, i.e. $\mu_\eta\leq 0$ for all convex entropies $\eta$. Kru\v{z}kov introduced this concept in \cite{kruzkov70} and showed that for any bounded initial condition $u_0(x)$ there exists a unique entropy solution. 

\subsection*{Finite-entropy solutions}

Here in contrast we consider weak solutions whose entropy productions do not necessarily have a sign. 
Such solutions are \emph{not} uniquely determined by their initial conditions: they will in general deviate from the unique entropy solution, and the present work addresses the question of estimating this deviation.
Our motivation comes from the study of large deviation principles for the hydrodynamic limit of asymmetric interacting particle systems \cite{kipnis-landim,varadhan04}, where it is crucial to control how much a general weak solution can deviate from the entropy solution.

 To give more details about this issue, we focus on a continuum variant introduced in  \cite{mariani10,bellettini-etal}. There, the question boils down to
describing the variational convergence ($\Gamma$-convergence) of functionals of the form 
\begin{align*}
E_\e(u)=\frac{1}{\e}\int \left[\e\partial_x u -\partial_x^{-1}(\partial_t u +\partial_x f(u))\right]^2 \, dx\, dt,
\end{align*}
in the regime $\e\to 0^+$. The same problem is considered in \cite{poliakovsky-viscosity} (with the motivation of providing a variational point of view on the vanishing viscosity method). Limits $u=\lim u_\e$ of sequences of bounded energy $E_\e(u_\e)\leq C$ are weak solutions of \eqref{eq:scl}, but not necessarily entropy solutions. They belong to the wider class that we will call here \emph{finite-entropy solutions}: bounded weak solutions of \eqref{eq:scl} such that
\begin{align}\label{eq:finiteentropy}
\mu_\eta \text{ is a Radon measure for all convex }\eta,
\end{align}
where $\mu_\eta$ is the entropy production defined in \eqref{eq:mueta}.
The conjectured limiting energy $E_0(u)$ is the negative part $\mu_-([0,T]\times\R)$ of their entropy production,  but a proof of this fact is still lacking. 

Specifically, the missing part is the upper bound: given a finite-entropy solution $u$, can one construct an approximating sequence $u_\e\to u$ in $L^1$ such that $\limsup E_\e(u_\e) \leq E_0(u) $ ? Very similar questions arise in relation with micromagnetics  models (the so-called Aviles-Giga energy), we refer to the introduction of \cite{lamy-otto} for more details.
What makes this question hard is the lack of fine knowledge on finite-entropy solutions. Unlike entropy solutions, they are not necessarily of bounded variation (BV). Only very recently E.~Marconi \cite{marconi-structure,marconi-rectif} proved that their entropy production is a one-dimensional rectifiable measure. This rectifiability result is a remarkable achievement, but it seems that solving the upper bound problem requires other new ideas. 

To the best of our knowledge, only two upper bound constructions are available in the literature, with restrictive assumptions on the finite-entropy solution $u$. The first construction in \cite{poliakovsky-viscosity} requires $u$ to be BV, and the approximating sequence is obtained by mollifying $u$ and using the fine properties of BV functions. The second construction in \cite{bellettini-etal}
%
%
 is based on approximation by vanishing viscosity, which converges in open regions where the entropy production is $\leq 0$.
 If regions of negative and positive entropy production are not \enquote{well separated} this construction breaks down, for want of a good estimate on the distance between $u$ and entropy solutions when the entropy production changes sign. 
 
 In this spirit, the only estimate \cite{lamy-otto} we are aware of is not homogeneous:
\begin{align*}
 \int_{[0,1]_t\times [-1,1]_x} \abs{u-u^{ent}}^4 \leq C\, \mu_+([0,2]_t\times [-2,2]_x)^\gamma\qquad\text{for some }\gamma\in (0,1),
\end{align*}
where $u^{ent}$ is the entropy solution with initial data $u_0$ and $|u|\leq 1$. If one applies (a rescaled version of) this estimate in small regions where $\mu_+$ is small, after summing over all regions the right-hand side may become very large because of the small exponent $\gamma$. As a consequence, this estimate cannot be used to remove the main restriction (that the regions where the entropy production has a constant sign must be well separated) in the approximation scheme of \cite{bellettini-etal}. One would rather need an estimate that is homogeneous, hence amenable to summing rescaled applications of it. 

In this work we propose a new approach towards such estimate, beginning with the distance of $u$ to entropic shocks: if a solution $u$ starts close to a shock and $\mu_+$ is small, then $u$ remains close to a shock, and this is quantified via a homogeneous estimate. 
More precisely, our main result  takes the form of an $L^2$ stability estimate for entropic shocks. For entropy solutions this question was adressed in \cite{leger,leger-vasseur} using relative entropy methods. Here we generalize their methods to solutions whose entropy production does not necessarily have a sign. Loosely stated, we prove (Theorem~\ref{t:stabestim})
\begin{align*}
\int |u(t)-\text{shock}|^2 \, dx \leq  \int |u(0)-\text{shock}|^2dx  + C \, \int_0^t\! \int \mu_+(dt,dx),
\end{align*}
where the shock at time $t$ is a shift of the initial shock. We also provide a generalization to any $BV$ initial data (Theorem~\ref{t:stabent}).

\subsection*{Strong and very strong traces}

As in \cite{leger-vasseur,krupa-vasseur,krupa-vasseur20}, in order to implement the relative entropy method we need to assume that $u$ has traces on Lipschitz curves, in a strong enough sense. From \cite{vasseur--traces}, it is known that finite-entropy solutions have traces which are reached strongly in $L^1$.  We call this the \emph{strong trace} property, precisely defined as follows. A bounded function 
$u\colon [0,T]\times\R\to\R$ satisfies the strong trace property if for any Lipschitz path $x\colon [0,T]\to\R$ there exist traces $t\mapsto u(t,x(t)\pm)$ on each side of $x(t)$, such that
\begin{align}\label{eq:ST}
\esslim_{y\to 0^+}\int_0^T \abs{u(t,x(t)\pm y)- u(t,x(t)\pm)}\, dt =0.
\end{align}
In \cite{vasseur--traces} entropy solutions are considered, but the proof there uses only a kinetic formulation which is also valid for finite-entropy solutions \cite{delellis-otto-westdickenberg}.
The results of \cite{vasseur--traces} also include traces along constant time lines, implying that (for an a.e. representative)
\begin{align}\label{eq:timecont}
[0,T]\ni t\mapsto u(t,\cdot) \in L^1_{loc}\quad\text{is continuous,}
\end{align}
whenever $u$ is a finite-entropy solution of \eqref{eq:scl}.

Unfortunately the strong trace property turns out not to be enough for our purposes, and as in \cite{leger-vasseur,krupa-vasseur,krupa-vasseur20} we will in fact require an even stronger property. We say that a bounded function $u\colon [0,T]\times\R\to \R$ satisfies the \emph{very strong trace} property if for any Lipschitz path $x\colon [0,T]\to\R$ there exist traces $t\mapsto u(t,x(t)\pm)$ such that (for an a.e. representative of $u$)
\begin{align}\label{eq:VST}
\esslim_{y\to 0^+} u(t,x(t)\pm y) = u(t,x(t)\pm) \qquad\text{for a.e. }t\in [0,T].
\end{align}
By dominated convergence the very strong trace property does imply the strong trace property.
Functions $u\in BV([0,T]\times\R)$ satisfy the very strong trace property, but it is not known whether finite-entropy solutions  satisfy it.

\subsection*{Stability of shocks in $L^2$ for finite-entropy solutions} 
We are now ready to state our main result, on the $L^2$ stability of an entropic shock wave $u^{shock}$ with initial datum 
\begin{align}\label{eq:u0shock}
u_0^{shock}(x)&= u_\ell \mathbf 1_{x< 0} +u_r \mathbf 1_{x>0},\qquad u_\ell > u_r,
\end{align}
that is, $u^{shock}(t,x)=u_0^{shock}(x-\sigma t)$, with shock speed $\sigma=(f(u_r)-f(u_\ell))/(u_r-u_\ell)$.

\begin{theorem} \label{t:stabestim}
Let $f\colon \R\to\R$ be such that $f''\geq \alpha>0$. Let $u\colon [0,T]\times\R\to\R$ be a bounded finite-entropy solution \eqref{eq:finiteentropy} of
\begin{align*}
&\partial_t u +\partial_x f(u)=0,\qquad u(0,x)=u_0(x).
\end{align*}
Assume  that $u$ satisfies the very strong trace property \eqref{eq:VST}.
Let $u^{shock}$ be the entropic shock wave with initial datum $u_0^{shock}$ \eqref{eq:u0shock}, and set $M=\sup_I f''$ and $S=\sup_I |f'|$, where $I=[\min(u_r,\inf u),\max(u_\ell,\sup u)]$.

\medskip

There exists a Lipschitz path $h\colon [0 , T] \to \R$ such that $h(0)=0$ and
\begin{align}\label{eq:stabestim}
\int_{ -R}^{R} \abs{u(t,x) - u^{shock}(t,x - h(t))}^2 dx &\leq \int_{-R-tS}^{R+tS} \abs{u_0 - u_0^{shock}}^2  dx \\
&\quad + C\frac{M^3}{\alpha^3}\, \mu_+([0,t]\times [-R-tS,R+tS]),\nonumber
\end{align}
 for all $t\in [0,T]$, all $R>0$ and some absolute constant $C>0$, where $\mu$ is the entropy production \eqref{eq:mu} associated with $\eta(t)=t^2/2$.
 
\medskip

In addition the drift $h$ is controlled by
\begin{align}\label{eq:estimdrift}
c\frac{\alpha}{M^2}(u_\ell-u_r)\int_0^t h'(\tau)^2\, d\tau & \leq \int_{-2St}^{2St}(u_0-u_0^{shock})^2\, dx \\
&\quad
+ \frac{M^3}{\alpha^3}\mu_+([0,t]\times [-2St,2St])\nonumber
\end{align}
for some absolute constant $c>0$ and all $t\in [0,T]$.
\end{theorem}

\begin{remark}\label{r:vst}
We were not able to remove the very strong trace assumption from this statement. In the proof it is used only in Lemma~\ref{l:gencharac} to establish that $u$ admits \emph{generalized characteristics}:
 for any $x_0\in \mathbb R$, there exists a Lipschitz curve $x\colon [0,T]\mapsto \mathbb R$ such that $x(0)=x_0$ and
\begin{align*}
x'(t)=\sigma(u(t,x(t)-),u(t,x(t)+)\qquad\text{for a.e. }t\in [0,T],
\end{align*}
where $\sigma(u_-,u_+)=(f(u_+)-f(u_-))/(u_+-u_-)$ when $u_-\neq u_+$, and $\sigma(u,u)=f'(u)$.
Other places where traces are needed require only the strong trace property \eqref{eq:ST}, satisfied by finite-entropy solutions.
\end{remark}

\begin{remark}\label{r:estimdrift}
The necessity of introducing a drift $h(t)$,  and the near-optimality of estimate \eqref{eq:estimdrift} when $\mu_+=0$ and $u_\ell=-u_r=1$, are proved in \cite[Proposition~1.2]{vasseur16}.
\end{remark}

To prove Theorem~\ref{t:stabestim} we adapt the relative entropy arguments used in \cite{leger,leger-vasseur,krupa-vasseur20}  (see also \cite{serre-vasseur,krupa-vasseur,serre-vasseur--review,vasseur16}). The relative entropy method was introduced in \cite{dafermos79,diperna79} to study the $L^2$ stability of smooth solutions among entropy solutions, and later refined in \cite{leger,leger-vasseur} to obtain the $L^2$ stability (up to a drift) of shock waves (see \cite{adimurthi14} for $L^p$ stability estimates up to a drift). This method is also relevant in the study of hydrodynamic limits for fluid equations \cite{vasseur08}.  The basic idea is that for any constant $v_0$, one has an identity of the form
\begin{align*}
\frac 12 \partial_t (u-v_0)^2 = \mu -\partial_x q(u;v_0).
\end{align*} 
Stability of the constant state $v_0$ when $\mu\leq 0$ then follows by integrating over $x\in\R$, provided $q(u;v_0)$ is nice enough (e.g. has compact support). In the case of finite-entropy solutions, one also has to take into account the contribution of $\mu_+$. But when studying the stability of a shock, one integrates $\partial_t (u-u_\ell)^2$ and $\partial_t (u-u_r)^2$ on two complementary half-lines, and boundary terms appear at the junction. 

The crucial remark used in \cite{leger,leger-vasseur,krupa-vasseur20} is that, if the initial shock is shifted by a well-chosen length $x(t)$, then the boundary terms combine into a nonpositive contribution. There are two cases to consider, depending on whether or not $u(t,\cdot)$ jumps at $x(t)$. At times $t$ where it does not jump, the situation is the same for entropy or finite-entropy solutions, and the ideas of \cite{leger,leger-vasseur,krupa-vasseur20} apply also in our case. But at times $t$ where it does jump, an entropy solution can only make a negative jump, while a finite-entropy solution can also make a positive jump. More precisely, denoting by $(u_-,u_+)$ the values of the jump of $u$, it is shown in \cite{leger,leger-vasseur} that the dissipation rate $D(u_-,u_+;u_\ell,u_r)$ coming from the boundary terms satisfies
\begin{align*}
D(u_-,u_+;u_\ell,u_r)\leq 0\qquad\text{whenever }u_-\geq u_+ \text{ and }u_\ell\geq u_r.
\end{align*}
To include finite-entropy solutions, we have to consider also what happens when $u_- < u_+$. One cannot expect the dissipation rate $D$ to remain $\leq 0$, but what we do show (see Proposition~\ref{p:Dbound}) is that its positive part is controlled by the entropy cost of the jump, in other words by $\mu_+$. This crucial observation enables us to adapt the techniques of \cite{leger,leger-vasseur,krupa-vasseur20} to our situation and to prove the stability estimate \eqref{eq:stabestim}. 
In fact we prove a sharper upper bound on $D$, thanks to which the  control \eqref{eq:estimdrift} on the drift $h(t)$ can then be obtained  as in \cite{krupa-vasseur20}.

\subsection*{Stability of entropy solutions with $BV$ initial data}

As a complement we provide a generalization of Theorem~\ref{t:stabestim} where the entropy solution $u^{shock}$ is replaced by any entropy solution with $BV$ initial data.
This relies, as in  \cite{krupa-vasseur,chen-krupa-vasseur}, on applying the techniques of Theorem~\ref{t:stabestim}'s proof to obtain estimates between $u$ and functions with a finite number of shocks. Each shock wave has to be shifted, and all shifts may be different.
The estimate \eqref{eq:estimdrift} on the drift of one single shock in Theorem~\ref{t:stabestim} is therefore replaced by an estimate on the $L^1$ distance between the \enquote{shifted} function and the actual entropy solution.

\begin{theorem}\label{t:stabent}
Let $f\colon \R\to\R$ be such that $f''\geq \alpha>0$. Let $u\colon [0,T]\times\R\to\R$ be a bounded finite-entropy solution \eqref{eq:finiteentropy} of
\begin{align*}
&\partial_t u +\partial_x f(u)=0,\qquad u(0,x)=u_0(x).
\end{align*}
Assume  that $u$ satisfies the very strong trace property \eqref{eq:VST}.
Let $\zeta$ be an entropy solution of \eqref{eq:scl} with initial datum $\zeta_0\in L^\infty\cap BV_{loc}(\mathbb R)$, and set $M=\sup_I f''$ and $S=\sup_I |f'|$, where $I=[\min(\inf\zeta_0,\inf u),\max(\sup\zeta_0,\sup u)]$.

\medskip

There exists $\tilde u\in L^\infty\cap BV_{loc}( [0,T]\times\R)\cap \mathrm{Lip}([0,T],L^1_{loc}(\R))$ such that
\begin{align}\label{eq:stabentL2}
\int_{-R}^{R}\abs{u(t,x)-\tilde u(t,x)}^2 dx &\leq \int_{-R-St}^{R+St}\abs{u_0-\zeta_0}^2 dx
\\
&\quad +C\frac{M^3}{\alpha^3}\mu_+([0,t]\times [-R-St,R+St]),\nonumber
\end{align}
and
\begin{align}\label{eq:stabentL1}
&\int_{-R}^R\abs{\tilde u(t,x)-\zeta(t,x)}\, dx \\
&\leq C
\frac{M}{\alpha^{\frac 12}}\sqrt{(D\zeta_0)_-([-R-St,R+St])}\sqrt{t}\nonumber\\
&\hspace{3em}\cdot
\sqrt{
\int_{-R-St}^{R+St}\abs{u_0-\zeta_0}^2 dx +C\frac{M^3}{\alpha^3}\mu_+([0,t]\times [-R-St,R+St])},\nonumber
\end{align}
for some absolute constant $C>0$, all $t\in [0,T]$ and all $R>0$. 
\end{theorem}

\begin{remark}
In the case $\zeta_0=u_0^{shock}$, Theorem~\ref{t:stabent} is a corollary of Theorem~\ref{t:stabestim} taking $\tilde u(t,x)=u_0^{shock}(t,x-h(t))$ so that \eqref{eq:stabentL2} is exactly \eqref{eq:stabestim}, and \eqref{eq:stabentL1} follows from \eqref{eq:estimdrift}.
\end{remark}

\begin{remark}
It is well-known that weak differentiability of order $s=1/3$ is critical for finite entropy solutions of \eqref{eq:scl} (see \cite{delellis-westdickenberg,golseperthame13,delellisignat15}). Somewhat interestingly this critical exponent also comes up in relation with Theorem~\ref{t:stabent}:
if one wishes to use Theorem~\ref{t:stabent} in order to estimate the distance of $u$ to the entropy solution starting at $u_0$ (when $u_0$ is not $BV$) in terms of $\mu_+$, it seems natural to consider $\zeta_0= u_0 * \rho_\e$ with $\rho_\e(x)=\e^{-1}\rho(x/\e)$ for some smooth kernel $\rho$ and $\e$ small enough so that $\int |u_0 -u_0 *\rho_\e|^2\, dx\lesssim \int \mu_+(dt,dx)$. The right-hand side of \eqref{eq:stabentL1} then puts forward the square root of the product
\begin{align*}
\int |(u_0 * \rho_\e)'|\, dx \cdot \int |u_0 -u_0 *\rho_\e|^2\, dx.
\end{align*}
If  $u_0$ enjoys some fractional derivability of order $s>0$ (e.g. of Besov $B^{s}_{2,\infty}$ or Sobolev $W^{s,2}$ type), the first factor is typically bounded by $\e^{s-1}$, the second by $\e^{2s}$, hence this product is bounded by $\e^{3s-1}$, and the exponent $s=1/3$ is critical.
\end{remark}

\subsection*{Outline}

The article is organized as follows. In section~\ref{s:D} we prove the new bound on the dissipation rate $D$ appearing in the relative entropy method. In section~\ref{s:proof} we recall and adapt the arguments of \cite{leger,leger-vasseur} to prove Theorem~\ref{t:stabestim}. In section~\ref{s:stabent} we prove Theorem~\ref{t:stabent}.

\subsection*{Acknowledgements} X.L. is partially supported by ANR project ANR-18-CE40-0023 and COOPINTER project IEA-297303.

\section{Upper bound on the dissipation rate $D$}\label{s:D}

We start by setting some notations. We denote by $\eta$, $q$ the entropy-flux pair given by
\begin{align*}
\eta (x) = \frac{x^2}{2} ,\quad q(x) = \int_0^x \eta'f',
\end{align*}
and by $\eta (\cdot \vert \cdot) , q(\cdot ; \cdot)$ the corresponding relative entropy-flux pair
\begin{align*}
\eta(x \vert a) &= \eta(x) - \eta(a) - \eta '(a) (x - a) = \frac{{(x - a)}^2}{2}\\
q(x ; a) &= q(x) - q(a) - \eta ' (a) (f(x) - f(a)).
\end{align*}
The propagation speed of a shock $(u_-,u_+)$ is constrained by the Rankine-Hugoniot condition:
\begin{align}\label{eq:sigma}
\sigma(u_-,u_+)=\frac{f(u_+)-f(u_-)}{u_+-u_-},
\end{align}
and by setting $\sigma(u,u)=f'(u)$ the function $\sigma$ is continuous on $\R^2$.
Given two shocks $(u_-,u_+)$ and $(u_\ell,u_r)$ we define the dissipation rate
\begin{align}\label{eq:D}
D(u_-,u_+;u_\ell, u_r) := q(u_+ ; u_r) - q(u_- ; u_\ell) - \sigma(u_-,u_+) \left( \eta(u_+|u_r)-\eta(u_-|u_\ell) \right).
\end{align}
As explained in the introduction, this corresponds to the boundary terms which arise when calculating 
\begin{align*}
\partial_t \int \eta\big(u(t,x)\big| u_0^{shock}(x-x(t))\big)\, dx,
\end{align*}
at times $t$ where $u(t,\cdot)$ has a jump $(u_-,u_+)$ at $x=x(t)$.

Our goal in this section is to compare the dissipation rate $D$ with the entropy cost of the jump $(u_-,u_+)$, given by
\begin{align}\label{eq:E}
E(u_-,u_+)&=q(u_+)-q(u_-)-\sigma(u_-,u_+)(\eta(u_+)-\eta(u_-)).
\end{align}
This formula corresponds to the fact that, if a solution $u$ has a jump $(u_-(t),u_+(t))$ along a curve $x(t)$, and is smooth everywhere else, then by the BV chain rule (see e.g. \cite[\S~3.10]{AFP}) the entropy production $\mu$ is given by
\begin{align*}
\mu(A) =\int \mathbf 1_{(t,x(t))\in A}\, E(u_-(t),u_+(t))\, dt.
\end{align*}
The main result of this section is the following.
\begin{proposition} \label{p:Dbound}
For $u_\ell\geq u_r$ and any $u_\pm\in\R$ we have
\begin{align*}
 D(u_-,u_+ ; u_\ell, u_r)
&\leq C_1\frac{M^3}{\alpha^3} \max( E(u_-,u_+), 0)\\ 
&\quad - C_2\,\alpha\,(u_\ell-u_r)\left[(u_\ell-u_-)^2+(u_r-u_+)^2 \right],
\end{align*}
for some absolute constants $C_1,C_2>0$ and $0<\alpha\leq M$ such that $\alpha \leq f''\leq M$ on the convex hull of $\lbrace u_-,u_+,u_\ell,u_r\rbrace$.
\end{proposition}

\begin{proof}[Proof of Proposition~\ref{p:Dbound}]
The case $u_-\geq u_+$ can be inferred from  the arguments in \cite[Lemma~4.1]{krupa-vasseur20}, only the case $u_-<u_+$ is really new. For the reader's convenience we include a proof in both cases.  Only the values of $f$ on the convex hull of $\lbrace u_-,u_+,u_\ell,u_r\rbrace$ play a role in this inequality, so we assume without loss of generality that $\alpha\leq f''\leq M$ on $\R$.

\medskip

\noindent\textbf{Case 1: $u_+\geq u_-$.}

\noindent
We split $D$ as
\begin{align}\label{eq:Dsplit}
D(u_-,u_+;u_\ell,u_r)&=E(u_-,u_+) + F_-(u_-,u_+;u_\ell,u_r) + D(u_-,u_-;u_\ell,u_r)\\
&=E(u_-,u_+) + F_+(u_-,u_+;u_\ell,u_r) + D(u_+,u_+;u_\ell,u_r)\nonumber
\end{align}
where
\begin{align*}
F_-(u_-,u_+;u_\ell,u_r)&=D(u_-,u_+;u_\ell,u_r)-E(u_-,u_+)-D(u_-,u_-;u_\ell,u_r)\\
&=\eta'(u_r)\left[ f(u_-)-f(u_+) +\sigma(u_-,u_+)u_+ -f'(u_-)u_- \right] \\
&\quad
+(\sigma(u_-,u_+)-f'(u_-))\big[
\eta(u_r)-\eta(u_\ell)-u_r\eta'(u_r) \\
&\hspace{14em} + u_\ell\eta'(u_\ell) - \eta'(u_\ell)u_-
\big], \\
F_+(u_-,u_+;u_\ell,u_r)&=D(u_-,u_+;u_\ell,u_r)-E(u_-,u_+)-D(u_+,u_+;u_\ell,u_r)\\
&=\eta'(u_\ell)\left[ f(u_-)-f(u_+) -\sigma(u_-,u_+)u_- +f'(u_+)u_+ \right] \\
&\quad
+(\sigma(u_-,u_+)-f'(u_+))\big[
\eta(u_r)-\eta(u_\ell)-u_r\eta'(u_r) \\
&\hspace{14em} + u_\ell\eta'(u_\ell) + \eta'(u_r)u_+,
\big].
\end{align*}
We also define
\begin{align*}
\Delta:=u_+-u_-\geq 0,\quad A_\pm:=u_\ell-u_\pm \geq  B_\pm:=u_r - u_\pm,
\end{align*}
and start by remarking that
\begin{align}
E(u_-,u_+)&= \frac \beta {12} \Delta^3\qquad\text{for some }\beta\in [\alpha,M], \label{eq:idE}\\
F_\pm(u_-,u_+;u_\ell,u_r)&=\mp \frac {\gamma_\pm} 4 \Delta (A_\pm^2-B_\pm^2) 
\qquad\text{for some }\gamma_\pm\in [\alpha,M],\label{eq:idFpm}\\
D(u_\pm,u_\pm;u_\ell,u_r)&\leq -\frac{\alpha}{6} (A_\pm^3-B_\pm^3). \label{eq:boundDu-} 
\end{align}
\medskip
\noindent\textit{Proof of \eqref{eq:idE}.} Recalling that $\eta(t)=t^2/2$ and $q'(t)=t f'(t)$ we have
\begin{align*}
E(u_-,u_+)&=q(u_+)-q(u_-)-\sigma(u_-,u_+)(\eta(u_+)-\eta(u_-))\\
&=\int_{u_-}^{u_+} tf'(t)\, dt  -\frac{u_+ +u_-}{2}\int_{u_-}^{u_+}f'(t)\, dt \\
&=\frac 14 (u_+-u_-)^2 \int_{-1}^1 s f'\left(\frac{u_++u_-}{2} +s \frac{u_+-u_-}{2}\right)\, ds \\
&=\frac 14 (u_+-u_-)^2 \int_{-1}^1 s \left[f'\left(\frac{u_++u_-}{2} +s \frac{u_+-u_-}{2}\right)-f'\left(\frac{u_+ +u_-}{2} \right)\right]\, ds \\
&=\frac 18 (u_+-u_-)^3\int_0^1 \int_{-1}^1 s^2 f''\left(\frac{u_++u_-}{2} +st \frac{u_+-u_-}{2}\right)\, ds\, dt.
\end{align*}
This last expression implies \eqref{eq:idE} since $\alpha\leq f''\leq M$  and $\int_0^1\int_{-1}^1 s^2 \, ds dt = 2/3$.\qed

\medskip
\noindent\textit{Proof of \eqref{eq:idFpm}.}
We have 
\begin{align*}
 F_-
&=\frac 12 (u_\ell -u_r)(u_\ell+u_r-2u_-)\frac{1}{u_+-u_-}\int_{u_-}^{u_+}(f'(t)-f'(u_-))\, dt
\\
&=\frac 12 (A_--B_-)(A_-+B_-)\int_{0}^{1}\left[f'(u_- + s(u_+-u_-))-f'(u_-)\right]\, ds
\\
&= \frac 12 \Delta (A_-^2-B_-^2)\int_0^1\int_0^1 t\,f''(u_- + st(u_+-u_-))\, ds\, dt,
\end{align*}
which gives \eqref{eq:idFpm} for $F_-$ since $\alpha\leq f''\leq M$  and $\int_0^1\int_0^1 t\, ds dt=1/2$. Similarly
\begin{align*}
F_+
&= -\frac 12 \Delta (A_+^2-B_+^2)\int_0^1\int_0^1 t\,f''(u_+ -st (u_+-u_-))\, ds\, dt
\end{align*}
which gives \eqref{eq:idFpm} for $F_+$.\qed

\medskip
\noindent\textit{Proof of \eqref{eq:boundDu-}.}
We have 
\begin{align*}
D(u,u;u_\ell,u_r)
&= \int_{u_r}^{u_\ell}tf'(t)\, dt - u_\ell \int_{u}^{u_\ell}f'(t)\, dt
-u_r\int_{u_r}^{u}f'(t)\, dt \\
&\quad 
+\frac 12 (u_\ell-u_r)(u_\ell+u_r-2u)f'(u)
\\
&= \int_{u}^{u_\ell} (t - u_\ell) f'(t)\, dt
+\int_{u_r}^{u} (t - u_r) f'(t)\, dt \\
&\quad
+ \left( \int_u^{u_\ell} (t-u_\ell)+\int_{u_r}^{u} (t-u_r) \right)f'(u)
\\
&=\int_{u}^{u_\ell}(t-u_\ell)(f'(t)-f'(u))\, dt 
+ \int_{u_r}^{u}(t-u_r)(f'(t)-f'(u))\, dt 
\\
&=\int_{u}^{u_\ell}(t-u_\ell)(t-u)\int_0^1 f''(u+s(t-u))\, ds\, dt\\
& \quad +\int_{u_r}^{u}(t-u_r)(t-u)\int_0^1 f''(u+s(t-u))\, ds\, dt.
\end{align*}
If $u\in [u_r,u_\ell]$ we see that $f''\geq \alpha $ implies
\begin{align*}
D(u,u;u_\ell,u_r)&\leq -\alpha\left(\int_{u}^{u_\ell}(u_\ell-t)(t-u)\, dt + \int_{u_r}^{u}(t-u_r)(u-t)\, dt\right) \\
&= -\frac \alpha 6 \left(A^3-B^3\right),
\end{align*}
where $A=u_\ell-u$ and $B=u_r-u$.
If $u\leq u_r$ we rewrite the above as
\begin{align*}
D(u,u;u_\ell,u_r)&=-\int_{u_r}^{u_\ell}(u_\ell-t)(t-u)\int_0^1 f''(u+s(t-u))\, ds\, dt\\
& \quad -\int_{u}^{u_r}(u_\ell-u_r)(t-u)\int_0^1 f''(u+s(t-u))\, ds\, dt,
\end{align*}
and if $u\geq u_\ell$ as
\begin{align*}
D(u,u;u_\ell,u_r)&=-\int_{u_\ell}^{u}(u_\ell-u_r)(u-t)\int_0^1 f''(u+s(t-u))\, ds\, dt\\
& \quad -\int_{u_r}^{u_\ell}(t-u_r)(u-t)\int_0^1 f''(u+s(t-u))\, ds\, dt,
\end{align*}
and in both cases we deduce again that \eqref{eq:boundDu-} is valid.\qed

\medskip
Combining \eqref{eq:Dsplit} with \eqref{eq:idE}-\eqref{eq:boundDu-} we obtain
\begin{align}\label{eq:Dbound1}
D(u_-,u_+;u_\ell,u_r)\leq \frac{\beta}{12}\Delta^3 \mp \frac{\gamma_\pm}{4}\Delta(A_\pm^2-B_\pm^2) -\frac{\alpha}{6}(A_\pm^3-B_\pm^3).
\end{align}
Since $\Delta\geq 0$, for any $A\geq B$ and $\gamma,\lambda>0$, by Young's inequality $ab\leq \frac 13 a^3+\frac 23 b^{\frac 32}$ ($a,b\geq 0)$ we have
\begin{align*}
\frac\gamma 4 \Delta \abs{A^2-B^2}\leq\frac{\gamma}{12\lambda^3}\Delta^3 + \frac \gamma 6 \lambda^{\frac 32} \abs{A^2-B^2}^{\frac 32}.
\end{align*}
From
\begin{align*}
\abs{A^2-B^2}^3&=(A-B)^2\abs{A^2-B^2}(A+B)^2 \leq 2 (A-B)^2 (A^2+B^2)^2 \\
&
\leq 8(A-B)^2 (A^2+AB+B^2)^2=8(A^3-B^3)^2,
\end{align*}
we deduce
\begin{align*}
\frac\gamma 4 \Delta \abs{A^2-B^2}\leq\frac{\gamma}{12\lambda^3}\Delta^3 +  {\sqrt 2} \frac\gamma 3 \lambda^{\frac 32} (A^3-B^3),
\end{align*}
and we see that choosing $\lambda^3=\alpha^2/(32 \gamma^2)$ leads to
\begin{align*}
\frac\gamma 4 \Delta \abs{A^2-B^2} \leq \frac{8}{3}\frac{\gamma^3}{\alpha^2} \Delta^3+\frac{\alpha}{12}(A^3-B^3).
\end{align*}
Plugging this into \eqref{eq:Dbound1} yields
\begin{align*}
D(u_-,u_+;u_\ell,u_r) &\leq \left(1+32\frac{\gamma_\pm^3}{\beta\alpha^2} \right)\frac{\beta}{12}\Delta^3 -\frac{\alpha}{12}(A_\pm^3-B_\pm^3).
\end{align*}
Recalling \eqref{eq:idE}-\eqref{eq:idFpm} this implies
\begin{align*}
D(u_-,u_+;u_\ell,u_r) &\leq C \frac{M^3}{\alpha^3} E(u_-,u_+) -\frac{\alpha}{12}(A_\pm^3-B_\pm^3),
\end{align*}
for $C=33$.
Remarking that 
\begin{align*}
A^3-B^3 &=(A-B)(A^2+AB+B^2)\geq (A-B)\frac{A^2+B^2}{2},\\
A_\pm -B_\pm&=u_\ell -u_r,
\end{align*}
we deduce
\begin{align*}
D(u_-,u_+;u_\ell,u_r) &\leq C \frac{M^3}{\alpha^3} E(u_-,u_+)\\
&\quad
-\frac{\alpha}{24}(u_\ell-u_r)\left[(u_\ell -u_-)^2+(u_r-u_+)^2\right],\nonumber
\end{align*}
which proves Proposition~\ref{p:Dbound} when $u_+\geq u_-$.

\medskip
\noindent\textbf{Case 2: $u_- > u_+$.}

\noindent
Following \cite{serre-vasseur} (see also \cite[Lemma~3.3]{krupa-vasseur20}) we write $D$ as a combination of integrals of the function
\begin{align*}
g(t)&=f(t)-\sigma(u_-,u_+) t - (f(u_+)-\sigma(u_-,u_+)u_+)\\
&=f(t)-\sigma(u_-,u_+) t - (f(u_-)-\sigma(u_-,u_+)u_-),
\end{align*}
where the second equality follows from the definition of $\sigma$.
Specifically we have
\begin{align*}
D(u_-,u_+;u_\ell,u_r)&=-\int_{u_r}^{u_+}g(t)\, dt + \int_{u_\ell}^{u_-} g(t)\, dt.
\end{align*}
As remarked in \cite{serre-vasseur}, since $u_+\leq u_-$ and $u_r\leq u_\ell$, this can be written as
\begin{align}\label{eq:idDg}
-D=\int_I \abs{g} +\int_J \abs{g},
\end{align}
where $I$ and $J$ are disjoint intervals such that 
\begin{align*}
I\cup J=([u_+,u_-]\cup [u_r,u_\ell])\setminus ([u_+,u_-]\cap [u_r,u_\ell]).
\end{align*}
We denote by $g_0$ the function $g$ in the special case $f(t)=f_0(t)=t^2/2$, i.e.
\begin{align*}
g_0(t)=\frac 12 (t-u_+)(t-u_-).
\end{align*}
Since $g-\alpha g_0$ is a convex function vanishing at $u_+$ and $u_-$, we have 
\begin{align*}
g(t)&\geq \alpha g_0(t) \quad\text{for }t\in \R\setminus [u_+,u_-],\\
g(t)&\leq \alpha g_0(t) \quad\text{for }t\in  [u_+,u_-].
\end{align*}
Therefore $\abs{g}\geq \alpha \abs{g_0}$ on $\R$ and one infers from \eqref{eq:idDg} that
\begin{align*}
-D\geq -\alpha D_0,
\end{align*}
where $D_0$ is the dissipation rate \eqref{eq:D} for $f(t)=f_0(t)=t^2/2$.  To conclude the proof of
Proposition~\ref{p:Dbound} it suffices to check that
\begin{align}\label{eq:boundD0}
-D_0(u_-,u_+;u_\ell,u_r) \geq c (u_\ell-u_r) \left[(u_\ell-u_-)^2 +(u_r-u_+)^2 \right],
\end{align}
for some absolute constant $c>0$. 
To this end we compute
\begin{align*}
\int_{u_+}^{u_r} g_0 &= \frac 16 (u_r-u_+)^3 +\frac{1}{4} (u_+-u_-)(u_r-u_+)^2,\\
\int_{u_-}^{u_\ell} g_0 & =\frac 16(u_\ell - u_-)^3+\frac 14 (u_--u_+)(u_\ell-u_-)^2,
\end{align*}
so that, setting $H=u_\ell-u_-$ and $K=u_r-u_+$, we find
\begin{align*}
-D_0 &= \frac 14 (u_--u_+)(H^2 +K^2)+\frac 16 (H^3-K^3).
\end{align*}
This implies
\begin{align*}
-D_0&=\frac 14 (u_\ell - u_r -H +K)(H^2+K^2) +\frac 16 (H^3-K^3) \\
&=\frac 14 (u_\ell - u_r)(H^2+K^2) -\frac 1{12}(H-K)^3
\end{align*}
As $H-K=u_\ell-u_r -(u_--u_+)\leq u_\ell-u_r$ we deduce
\begin{align*}
-D_0&\geq \frac 14 (u_\ell-u_r)(H^2+K^2)-\frac 1{12}(u_\ell-u_r)(H-K)^2 \\
& =\frac 14 (u_\ell - u_r)\left(H^2+K^2-\frac 13(H-K)^2\right)\\
&\geq \frac 1{12}(u_\ell - u_r)(H^2+K^2),
\end{align*}
proving \eqref{eq:boundD0} with $c=1/12$.
\end{proof}

\section{The stability estimate for shocks}\label{s:proof}

We follow \cite{leger,leger-vasseur,serre-vasseur}, where $u$ is an entropy solution, and explain how their methods adapt to our more general situation. Our goal is to control the increase of 
\begin{align}\label{eq:F}
F(t)= \int_{-R+St}^{R-St}\eta\big(u(t,x)\big|u_0^{shock}(x-x(t))\big)\, dx,
\end{align}
for a well-chosen Lipschitz path $x(t)$ and $R\geq St$.

First we recall properties of the traces of $u$ along Lipschitz curves, which require only the strong trace property (and are thus valid for finite-entropy solutions). 

\begin{lemma}\label{l:RH}
Let $u$ be a weak bounded solution of \eqref{eq:scl} satisfying the strong trace property \eqref{eq:ST}. Let $x\colon [0,T]\to\R$ be a Lipschitz path, and $u(t,x(t)\pm)$ the traces of $u$ along $x(t)$. Then for almost every $t\in [0,T]$ we have the Rankine-Hugoniot relation
\begin{align}\label{eq:RH}
f(u(t,x(t)+))-f(u(t,x(t)-)=x'(t)(u(t,x(t)+)-u(t,x(t)-)).
\end{align}
\end{lemma}
\begin{proof}[Proof of Lemma~\ref{l:RH}]
This is proved in \cite[Lemma~6]{leger-vasseur}. We sketch the proof for the reader's convenience.
The Rankine Hugoniot relation \eqref{eq:RH} follows from testing the equation
\begin{align*}
\partial_t u +\partial_x f(u) =0,
\end{align*}
against a test function $\chi$ of the form
\begin{align*}
&\chi(t,\xi)=\psi(t)\left(\Phi_\e(\xi-x(t))+\Phi_\e(x(t)-\xi)-1\right),\\
\text{where }&\mathbf 1_{y<-\e} \leq \Phi_\e(y) \leq \mathbf 1_{y<0},
\end{align*}
The strong trace property \eqref{eq:ST} ensures convergence, as $\e\to 0$, to \eqref{eq:RH} tested against $\psi(t)$.
\end{proof}

Next we establish a formula for the variations of quantities of the form
\begin{align*}
t\mapsto \int_{x(t)}^{y(t)}\eta(u(t,x)|v_0)\, dx,
\end{align*}
for some constant $v_0$.

\begin{lemma}\label{l:var}
Let $u$ be a finite-entropy \eqref{eq:finiteentropy} solution of \eqref{eq:scl}. Let $y,z\colon[0,T]\to\R$ be Lipschitz paths, let $0\leq t_1<t_2\leq T$ and assume that
\begin{align*}
y(\tau) < z(\tau)\qquad\forall \tau\in (t_1,t_2).
\end{align*}
For any $v_0\in\R$, 
we have
\begin{align}\label{eq:var}
&\int_{y(t_2)}^{z(t_2)}\eta(u(t_2,\xi)|v_0)\, d\xi -\int_{y(t_1)}^{z(t_1)}\eta(u(t_1),\xi)|v_0)\, d\xi \\
&=\int \mathbf 1_{ t_1 < \tau < t_2,\, y(\tau)<\xi<z(\tau)}\; \mu(d\tau,d\xi)
\nonumber\\
&\quad + \int_{t_1}^{t_2} \left[ q(u(\tau,y(\tau)+);v_0) -y'(\tau)\eta(u(\tau,y(\tau)+)|v_0) \right] \, d\tau \nonumber\\
&\quad - \int_{t_1}^{t_2} \left[ q(u(\tau,z(\tau)-);v_0) -z'(\tau)\eta(u(\tau,z(\tau)-)|v_0) \right] \, d\tau.
\nonumber
\end{align}
\end{lemma}
\begin{proof}[Proof of Lemma~\ref{l:var}]
The proof is essentially the same as e.g. \cite[Lemma~6]{leger-vasseur}, see also \cite[Lemma~2.4]{krupa-vasseur}. We sketch it here for the reader's convenience, the only difference being that we keep the terms involving $\mu$. We may assume without loss of generality that $y<z$ in $[t_1,t_2]$ (otherwise consider instead $[t_1+\delta,t_2-\delta]$ and let $\delta\to 0^+$ at the end).

We test the identity
\begin{align*}
\partial_t \eta(u|u_0) +\partial_x q(u|u_0)=\mu ,
\end{align*}
against a test function $\chi$ of the form
\begin{align*}
&\chi(\tau,\xi)=\psi_\e(\tau)\left(\Phi_\e(y(\tau)-\xi) -\Phi_\e(\xi-z(\tau))\right),\\
\text{where }& \mathbf 1_{t_1+\e<\tau<t_2-\e}\leq \psi_\e(\tau) \leq \mathbf 1_{t_1<\tau<t_2},\quad \mathbf 1_{x<-\e} \leq \Phi_\e(x) \leq \mathbf 1_{x<0}.
\end{align*}
and obtain \eqref{eq:var} as $\e\to 0^+$, thanks to the strong trace property \eqref{eq:ST} and the time-continuity property \eqref{eq:timecont}.
\end{proof}

 Using Lemma~\ref{l:var} we will obtain a formula for the variations of $F(t)$ \eqref{eq:F}, and thanks to Lemma~\ref{l:RH} we will see that at any time $t$ where $u(t,\cdot)$ has a jump $(u_-,u_+)$ at $x(t)$, the increase of $F(t)$ is controlled by $\mu$ plus the dissipation rate $D$, which owing to Proposition~\ref{p:Dbound} is in turn controlled by $\mu_+$. Note that so far this is valid for any Lipschitz curve $x(t)$. However, in order to control the increase of $F(t)$ at times $t$ where $u(t,\cdot)$ does not jump at $x(t)$, we need to constrain $x'(t)$. The next lemma gives us a tool to do so. This is the only place where we require the  very strong trace property.

\begin{lemma}\label{l:gencharac}
Let $u$ be a bounded finite-entropy \eqref{eq:finiteentropy} solution of \eqref{eq:scl} and assume that $u$ satisfies the very strong trace property \eqref{eq:VST}. Then for any $x_0\in\R$ there exists a generalized characteristic of $u$ starting at $x_0$, that is, a Lipschitz path
 $x\colon [0,T]\to\R$ such that $x(0)=x_0$ and
\begin{align}\label{eq:gencharac}
x'(t)=\sigma(u(t,x(t)-),u(t,x(t)+)\qquad\text{for a.e. }t\in [0,T]
\end{align}
where $u(t,x(t)\pm)$ denote the traces of $u$ along $x(t)$ and $\sigma$ is the shock speed \eqref{eq:sigma}
\end{lemma}
\begin{proof}[Proof of Lemma~\ref{l:gencharac}]
This is proved e.g. in \cite[Proposition~1]{leger-vasseur}. The path $x$ is obtained as a limit of paths $x_k(t)$ solving $x_k'(t)=\Phi_k(t,x_k(t))$, where $\Phi_k$ is a mollification (with respect to the $x$ variable) of $f'\circ u$. The very strong trace property then ensures that $x$ satisfies
\begin{align*}
\min\lbrace f'(u(t,x(t)\pm)) \rbrace \leq x'(t) \leq \max \lbrace f'(u(t,x(t)\pm)) \rbrace\quad\text{for a.e. }t\in[0,T],
\end{align*}
so that $x'(t)=f'(u(t,x(t)))$ for a.e. $t$ where there is no jump, i.e. $u(t,x(t)-)=u(t,x(t)+=u(t,x(t))$, and at jump points \eqref{eq:gencharac} follows from \eqref{eq:RH}.
\end{proof}
%

We are now ready to prove our main result.

\begin{proof}[Proof of Theorem~\ref{t:stabestim}]
We apply Lemma~\ref{l:gencharac} and let the Lipschitz path $x\colon [0,T]\to\R$ be a generalized characteristic starting at $0$, i.e. $x(0)=0$ and 
\begin{align}\label{eq:x'}
x'(t)=\sigma(u_-(t),u_+(t))\quad\text{for a.e. }t\in [0,T],
\end{align}
where $u_\pm(t)=u(t,x(t)\pm)$ denote the traces of $u$ along $x(t)$.

Let $R>0$ and $F(t)$ defined by \eqref{eq:F} for all $t \leq  R/S$. We assume without loss of generality that $R\geq ST$ (otherwise replace $T$ by $R/S$).
Consider the time
\begin{align*}
t_*= \sup \left\lbrace t\in [0,T]\colon -R+St<x(t)<R-St\right\rbrace.
\end{align*}
By definition of $S=\sup_I |f'|$  we
 know that $\abs{x'}\leq S$ and deduce
\begin{align*}
& -R+St < x(t) < R-St\qquad\forall t\in [0,t_*),\\
\text{and }& x(t)\in (-\infty,-R+St]\cup [R-St,\infty)\qquad\forall t\geq t_*.
\end{align*}
For $t\in [0,t_*]$ we have
\begin{align*}
F(t)=\int_{-R+St}^{x(t)}\eta(u(t,x)|u_\ell)\, dx  + \int_{x(t)}^{R-St}\eta(u(x,t)|u_r)\, dx.
\end{align*}
We apply the variation formula \eqref{eq:var} to compute $F(t)-F(0)$. Note that  the identity
\begin{align*}
\frac{q(u;v)}{\eta(u|v)}=\frac{2}{(u-v)^2}\int_v^u (t-v)f'(t)\, dt
=2\int_0^1 s\, f'(us+v(1-s))\, ds.
\end{align*}
and the definition of $S$ ensure that
\begin{align}\label{eq:qSeta}
\abs{q(u;v)}\leq S \eta(u|v)\qquad\forall u,v\in I=[\min(u_\ell,\inf u),\max(u_r,\sup u)].
\end{align}
As a consequence, whenever $y'(t)=S$ or $z'(t)=-S$ the corresponding term in the right-hand side of \eqref{eq:var} gives a nonpositive contribution, and we deduce
\begin{align*}
F(t)-F(0) &\leq \mu_+(B_{R,t}^S) \\
&\quad +\int_0^{t}\big[q(u_+(\tau);u_r)-q(u_-(\tau);u_\ell)\\
&\hspace{5em} -x'(\tau)\left(\eta(u_+(\tau)|u_r)-\eta(u_-(\tau)|u_\ell\right) \big]\,d\tau,
\end{align*}
where
\begin{align*}
B_{R,t}^S &=\left\lbrace (\tau,\xi)\colon 0<\tau<t,-R+\tau S <\xi <R-\tau S\right\rbrace.
\end{align*}
Recalling \eqref{eq:x'} that $x'=\sigma(u_-,u_+)$ a.e. in $[0,T]$, we recognize the dissipation rate $D$ \eqref{eq:D} and rewrite the above as
\begin{align}\label{eq:boundF1}
F(t)-F(0)&\leq \mu_+(B_{R,t}^S) +\int_0^t D(u_-(\tau),u_+(\tau);u_\ell,u_r)\, d\tau.
\end{align}
Since
\begin{align*}
\sigma(u_-,u_+)-\sigma(u_\ell,u_r)
=\int_0^1 \left[ f'(tu_-+(1-t)u_+) -  f'(tu_\ell +(1-t) u_r)\right] dt.
\end{align*}
and $f'$ is $M$-Lipschitz on $I$ we have
\begin{align*}
\abs{\sigma(u_-,u_+)-\sigma(u_\ell,u_r)}\leq \frac M2 \left( \abs{u_\ell-u_-}+\abs{u_r-u_+}\right),
\end{align*}
so using the upper bound on $D$ provided by Proposition~\ref{p:Dbound} we deduce from \eqref{eq:boundF1} that
\begin{align*}
F(t)-F(0)&\leq \mu_+(B_{R,t}^S) +C_1\frac{M^3}{\alpha^3}\int_0^t \max(E(\tau),0)\, d\tau \\
&\quad -C_2(u_\ell-u_r)\frac{\alpha}{M^2}\int_0^t(x'(\tau)-\sigma)^2\, d\tau,
\end{align*}
where $E(\tau)$ is the entropy cost of the jump $(u_-(\tau),u_+(\tau))$ \eqref{eq:E}.
From the characterization \cite{lecumberry,delellis-otto-westdickenberg} of the one-dimensional part of $\mu$ we have
\begin{align*}
\mu_{+\lfloor\lbrace (\tau,x(\tau))\rbrace} =\max(E(\tau),0)\, d\tau,
\end{align*}
and obtain
\begin{align}\label{eq:Fboundt*}
F(t)&\leq F(0) +  C \frac{M^3}{\alpha^3} \mu_+(B_{R,t}^S) \\
&\quad
-C_2(u_\ell-u_r)\frac{\alpha}{M^2}\int_0^t(x'(\tau)-\sigma)^2\, d\tau, \qquad\forall t\in [0,t_*].\nonumber
\end{align}
Now for $t\in [t_*,T]$ we have
\begin{align*}
F(t)=\int_{-R+St}^{R-St}\eta(u(x,t)|v_0)\, dx,
\end{align*}
where $v_0=u_r$ or $u_\ell$. Therefore, applying \eqref{eq:var} and remarking again that the terms involving $y'(t)=S$ and $z'(t)=-S$ give nonpositive contributions, we obtain
\begin{align*}
F(t)-F(t_*)&\leq \mu_+(B^S_{R,t}\setminus B^S_{R,t_*})\qquad\forall t\in [t_*,T].
\end{align*}
Combining this with the estimate obtained in $[0,t_*]$ and recalling the definition \eqref{eq:F} of $F(t)$ we deduce that
\begin{align*}
\int_{-R+tS}^{R-tS} \abs{u(x,t)-u_0^{shock}(x-x(t))}^2\, dx &\leq \int_{-R}^R\abs{u_0-u_0^{shock}}^2\, dx \\
& \quad +C\frac{M^3}{\alpha^3}\mu_+(B_{R,t}^S)\qquad\forall t\in [0,T].
\end{align*}
Provided we set $h(t)=x(t)-\sigma\, t$ and replace $R$ by $R+St$, this implies our main result \eqref{eq:stabestim} since $B_{R+St,t}^S\subset [0,t]\times [-R-St,R+St]$.

To prove estimate \eqref{eq:estimdrift} on $h(t)$, simply remark that \eqref{eq:Fboundt*} readily implies that
\begin{align*}
c\frac{\alpha}{M^2} (u_\ell-u_r)\int_0^t h'(\tau)^2\, d\tau \leq 
\int_{-R}^{R}(u_0-u_0^{shock})^2\, dx 
+ \frac{M^3}{\alpha^3}\mu_+(B_{R,t}^S)
\end{align*}
for some absolute constant $c>0$, provided $-R+St<x(t)<R-St$. Since we know that $\abs{x(t)}\leq St$ we may choose $R=2St$, and deduce \eqref{eq:estimdrift}.
\end{proof}

\section{The stability estimate for $BV$ initial data}\label{s:stabent}

This section is dedicated to Theorem~\ref{t:stabent}'s proof, following the scheme introduced in \cite{krupa-vasseur}.
It is based on considering initial conditions $\zeta_0$ with a finite number of entropic shocks at
 $x_1^0 <\cdots < x_N^0$, of the form
\begin{align}\label{eq:zeta0finite}
&\zeta_0 (x)=v^0_0(x)\mathbf 1_{x<x^0_1} +v^0_1(x)\mathbf 1_{x^0_1<x<x^0_2} +\cdots +v^0_N(x)\mathbf 1_{x>x^0_N},\\
&\text{where }v^0_0 \geq v^0_1 \geq \cdots \geq v^0_N\text{ are Lipschitz nondecreasing functions on }\mathbb R.\nonumber
\end{align}
Since $\zeta_0$ depends only on the restrictions of $v^0_j$ to pairwise disjoint intervals, we may  assume that their extensions to $\R$, in addition to being Lipschitz and nondecreasing, satisfy
\begin{align*}
&\inf\zeta_0 \leq v^0_j \leq \sup\zeta_0 \quad
\text{and}\quad 0\leq v^0_{j-1}-v^0_j \leq d^0_j\qquad\text{on }\R,
\end{align*}
for all $j\in\lbrace 0,\ldots,N\rbrace$, where $d^0_j:=(v^0_{j-1}-v^0_j)(x_0^j)$ is the schock amplitude of $\zeta_0$ at $x_j^0$. Note that with these notations the negative part of $D\zeta_0$ is given by 
$(D\zeta_0)_- = \sum_j d_j^0\delta_{x_j^0}$.

The function $\tilde u$ appearing in Theorem~\ref{t:stabent} is then going to be piecewise equal to the entropy solutions $v_j$ of \eqref{eq:scl} with initial data $v_j(0,x)=v_j^0(x)$. Since $f$ is convex and the $v_j^0$ are Lipschitz nondecreasing, $v_j$ is directly obtained by the method of characteristics and we have that
\begin{align*}
v_j\text{ is a Lipschitz solution of \eqref{eq:scl} and }\partial_x v_j\geq 0\qquad\forall j\in \lbrace 0,\ldots,N\rbrace.
\end{align*}
Moreover the entropy solutions $v_j$ are ordered as their initial data
\begin{align*}
\sup \zeta_0\geq v_0\geq v_1\geq  \cdots \geq v_N \geq \inf \zeta_0,
\end{align*}
and their differences have the following nonincreasing property: for any time $t>0$ and $R>0$, and intervals $I=[-R,R]$, $I_0=[-R-St,R+St]$ we have
\begin{align}\label{eq:vjdiff}
\sup_{I} (v_{j-1}(t,\cdot)-v_j(t,\cdot)) \leq \sup_{I_0}(v^0_{j-1}-v_j^0)\leq d_j^0.
\end{align}
This last assertion follows from the method of characteristics: for all $x\in I$ there are $y\leq z\in I_0$ such that
\begin{align*}
v_{j-1}(t,x)-v_j(t,x)=v^0_{j-1}(y)-v^0_j(z) \leq v^0_{j-1}(z)-v^0_j(z).
\end{align*}

We construct the function $\tilde u$ by shifting the shocks between the functions $v_j$. 
To deal with shocks between nondecreasing functions (instead of constant functions as in Theorem~\ref{t:stabestim}), we  need an equivalent of Lemma~\ref{l:var} where the constant $v_0$ is replaced by a Lipschitz solution $v$ of \eqref{eq:scl} with $\partial_x v\geq 0$. This corresponds to \cite[Lemma~2.4]{krupa-vasseur} which we transpose here to our setting:

\begin{lemma}\label{l:varnondec}
Let $u$ be a finite-entropy \eqref{eq:finiteentropy} solution of \eqref{eq:scl}. Let $y,z\colon[0,T]\to\R$ be Lipschitz paths, let $0\leq t_1<t_2\leq T$ and assume that
\begin{align*}
y(\tau) < z(\tau)\qquad\forall \tau\in (t_1,t_2).
\end{align*}
For any Lipschitz solution $v$ of \eqref{eq:scl} such that $\partial_x v\geq 0$
we have
\begin{align}\label{eq:varnondec}
&\int_{y(t_2)}^{z(t_2)}\eta(u(t_2,\xi)|v(t_2,\xi))\, d\xi -\int_{y(t_1)}^{z(t_1)}\eta(u(t_1),\xi)|v(t_1,\xi))\, d\xi \\
&\leq\int \mathbf 1_{ t_1 < \tau < t_2,\, y(\tau)<\xi<z(\tau)}\; \mu(d\tau,d\xi)
\nonumber\\
&\quad + \int_{t_1}^{t_2} \left[ q(u(\tau,y(\tau)+);v(\tau,y(\tau))) -y'(\tau)\eta(u(\tau,y(\tau)+)|v(\tau,y(\tau))) \right] \, d\tau \nonumber\\
&\quad - \int_{t_1}^{t_2} \left[ q(u(\tau,z(\tau)-);v(\tau,z(\tau))) -z'(\tau)\eta(u(\tau,z(\tau)-)|v(\tau,z(\tau))) \right] \, d\tau.
\nonumber
\end{align}
\end{lemma}
\begin{proof}[Proof of Lemma~\ref{l:varnondec}]
Since $v$ is Lipschitz we may use the chain rule and find
\begin{align*}
\partial_t\eta(u\vert v ) +\partial_x q(u\vert v) 
&= \mu -\eta''(v)\partial_x v \left[f(u)-f(v)-f'(v)(u-v) \right]\\
&\leq \mu,
\end{align*}
because $f$ is convex and $\eta''(v)\partial_x v\geq 0$. Then \eqref{eq:varnondec} follows as in the proof of Lemma~\ref{l:var}.
\end{proof}

Equipped with Lemma~\ref{l:varnondec}, the construction and estimates of $\tilde u$ become very similar to the proof of Theorem~\ref{t:stabent} as it simply consists in shifting the shocks along generalized characteristics of $u$. One small technical difficulty is that the curves may merge, or cross the bounds of integration. We start by considering the simplest setting where no merging nor crossing happens:

\begin{proposition}\label{p:stabfinite}
Let $t_1<t_2$ and $u\colon [t_1,t_2]\times\R\to\R$ be a bounded finite-entropy solution \eqref{eq:finiteentropy} of \eqref{eq:scl}. Let $x_1,\ldots, x_N\colon [t_1,t_2]\to\R$ be generalized characteristics of $u$, that is,
\begin{align*}
x_j'(t)=\sigma(u(t,x_j(t)-),u(t,x_j(t)+)\qquad\text{for a.e. }t\in [t_1,t_2].
\end{align*}
Assume that $R>0$ is such that for all $\tau\in (t_1,t_2)$,
\begin{align*}
-R-S(t_2-\tau):=x_0(\tau) < x_1(\tau) <\cdots <x_N(\tau)<x_{N+1}(\tau):=R+S(t_2-\tau).
\end{align*}
Then, for any $v_0\geq \cdots \geq v_N$ Lipschitz solutions of \eqref{eq:scl} such that $\partial_x v_j\geq 0$ for $j=0,\ldots,N$, and $\tilde u(t,x)$ defined by
\begin{align*}
\tilde u(t,x)=v_0(t,x)\mathbf 1_{x<x_1(t)} + v_1(t,x) \mathbf 1_{x_1(t)<x<x_2(t)} +\cdots +v_N(t,x)\mathbf 1_{x> x_N(t)},
\end{align*}
we have
\begin{align}\label{eq:estimL2utilde}
&\int_{x_0(t_2)}^{x_{N+1}(t_2)}\abs{u(t_2,x)-\tilde u(t_2,x)}^2 dx
-  \int_{x_0(t_1)}^{x_{N+1}(t_1)}\abs{u(t_1,x)-\tilde u(t_1,x)}^2  dx
\\
&\leq \Lambda(t_1,t_2):= C\frac{M^3}{\alpha^3}\mu_+((t_1,t_2)\times (x_0(t_1),x_{N+1}(t_1))),\nonumber
\end{align}
and, for any entropy solution $u^{ent}$ of \eqref{eq:scl},
\begin{align}\label{eq:estimL1utilde}
&\int_{x_0(t_2)}^{x_{N+1}(t_2)}\abs{\tilde u(t_2,x)-u^{ent}(t_2,x)}\, dx
-  \int_{x_0(t_1)}^{x_{N+1}(t_1)}\abs{\tilde u(t_1,x)-u^{ent}(t_1,x)}\, dx\\
&\leq 
\sum_{j=1}^N \int_{t_1}^{t_2} \sqrt{v_{j-1}(t,x_j(t))-v_j(t,x_j(t))}\, \theta_j(t) \, dt,\nonumber
\end{align}
where the functions $\theta_j$ satisfy
\begin{align}\label{eq:thetaj}
\frac{\alpha}{C}\sum_{j=1}^N \int_{t_1}^{t_2}\theta_j(t)^2\, dt \leq \int_{x_0(t_1)}^{x_{N+1}(t_1)}\abs{u(t_1,x)-\tilde u(t_1,x)}^2  dx + \Lambda(t_1,t_2),
\end{align}
and $C>0$ is an absolute constant.
\end{proposition}

\begin{remark}
In the case of one shock between constant functions, \eqref{eq:estimL2utilde} corresponds to \eqref{eq:stabestim} while \eqref{eq:estimL1utilde} can be inferred from \eqref{eq:estimdrift} and Kru\v{z}kov's $L^1$ stability estimate \cite[Theorem~6.2.3]{dafermos}.
\end{remark}

\begin{remark}\label{r:Jutilde}
The right-hand side of \eqref{eq:estimL1utilde} can be written more compactly as 
\begin{align*}
\int_{J_{\tilde u}}|[\tilde u]|^{\frac 12} \theta \, |\nu_x|\,d\mathcal H^1,
\end{align*}
where $J_{\tilde u}=\bigcup_j \lbrace (t,x_j(t)\rbrace_{t\in [t_1,t_2]}$ is the jump set of $\tilde u$ with normal vector $\nu=(\nu_t,\nu_x)$, 
$[\tilde u]=\tilde u_+-\tilde u_-$ denotes the jump of $\tilde u$ along that jump set, 
and the function $\theta$ is defined on $J_{\tilde u}$ by $\theta(t,x_j(t))=\theta_j(t)$.
 Then \eqref{eq:thetaj} translates into
\begin{align*}
\frac{\alpha}{C}\int_{J_{\tilde u}}\theta^2\, |\nu_x|\,d\mathcal H^1 \leq \int_{x_0(t_1)}^{x_{N+1}(t_1)}\abs{u(t_1,x)-\tilde u(t_1,x)}^2  dx + \Lambda(t_1,t_2).
\end{align*}
\end{remark}

\begin{proof}[Proof of Proposition~\ref{p:stabfinite}]
For all $j\in\lbrace 0,\ldots,N\rbrace$ we  apply Lemma~\ref{l:varnondec} to the paths $y=x_j$, $z=x_{j+1}$ and Lipschitz nondecreasing entropy solution $v_j$:
\begin{align*}
&\int_{x_j(t_2)}^{x_{j+1}(t_2)} \vert u(t_2,x) - \tilde{u}(t_2,x) \vert^2 dx - \int_{x_j(t_1)}^{x_{j+1}(t_1)} \vert u(t_1,x) - \tilde{u}(t_1,x) \vert^2 dx \\
& \leq \mu_+ (\{ t_1 < \tau < t_2, x_j(t) < \xi < x_{j+1} (t) \}) \\
&  \quad + \int_{t_1}^{t_2}\big[ q(u(\tau , x_j(\tau)+) ; v(\tau , x_j(\tau))) \\
&\hspace{5em} - x_j'(\tau) \eta(u(\tau,x_j(\tau)+) \vert v(\tau,x_j(\tau)))\big] \,d\tau  \\
& \quad - \int_{t_1}^{t_2} \big[ q(u(\tau , x_{j+1}(\tau)-) ; v(\tau , x_{j+1}(\tau))) \\
&\hspace{5em}- x_{j+1}'(\tau) \eta(u(\tau,x_{j+1}(\tau)-) \vert v(\tau,x_{j+1}(\tau)))\big]\, d\tau.
\end{align*}
Summing all these inequalities, we deduce
\begin{align*}
&\int_{x_0(t_2)}^{x_{N+1}(t_2)}\abs{u(t_2,x)-\tilde u(t_2,x)}^2 dx
-  \int_{x_0(t_1)}^{x_{N+1}(t_1)}\abs{u(t_1,x)-\tilde u(t_1,x)}^2  dx\\
&\leq \mu_+((t_1,t_2)\times(x_0(t_1),x_{N+1}(t_1)) \\
&\quad
+\sum_{j=1}^N \int_{t_1}^{t_2} D(u(t,x_j(\tau)-),u(\tau,x_j(\tau)+);v_{j-1}(\tau,x_j(\tau)),v_j(\tau,x_j(\tau)))\, d\tau
\end{align*}
Here we discarded the boundary terms involving the paths $x_0$ and $x_{N+1}$ as they give nonpositive contributions thanks to \eqref{eq:qSeta}.
Since $v_0 \geq \cdots \geq v_N$  we may apply Proposition~\ref{p:Dbound} to obtain, as for \eqref{eq:Fboundt*} in the proof of Theorem~\ref{t:stabestim},
\begin{align*}
&\int_{x_0(t_2)}^{x_{N+1}(t_2)}\abs{u(t_2,x)-\tilde u(t_2,x)}^2 dx
-  \int_{x_0(t_1)}^{x_{N+1}(t_1)}\abs{u(t_1,x)-\tilde u(t_1,x)}^2  dx\\
&\leq \Lambda(t_1,t_2) 
-\frac{\alpha}{C}\sum_{j=1}^N \int_{t_1}^{t_2} (v_{j-1}(\tau,x_j(\tau))-v_j(\tau,x_j(\tau))\, s_j(\tau)^2
\, d\tau,
\end{align*}
where
\begin{align}\label{eq:defsj}
s_j(t)=x_j'(t)-\sigma(v_{j-1}(t,x_j(t)),v_j(t,x_j(t))).
\end{align}
This implies in particular \eqref{eq:estimL2utilde}. Note for later use that it also implies
\begin{align}\label{eq:estimsj}
&\frac{\alpha}{CM^2}\sum_{j=1}^N \int_{t_1}^{t_2} (v_{j-1}(\tau,x_j(\tau))-v_j(\tau,x_j(\tau)))\, s_j(\tau)^2
\, d\tau \\
&\leq \int_{x_0(t_1)}^{x_{N+1}(t_1)}\abs{u(t_1,x)-\tilde u(t_1,x)}^2  dx +\Lambda(t_1,t_2).
\nonumber
\end{align}
Now we turn to the proof of \eqref{eq:estimL1utilde}. For any convex entropy $\tilde\eta$ and associated flux $\tilde q$, we compute
\begin{align*}
\tilde\mu_{\tilde\eta}:=\partial_t\tilde\eta(\tilde u) +\partial_x\tilde q(\tilde u)
\end{align*}
Note that $\tilde u$ need not be a solution of \eqref{eq:scl}, but we can still compute this entropy production. Since $\tilde u$ is a Lipschitz solution of \eqref{eq:scl} outside of the Lipschitz curves $x_j$, the $BV$ chain rule ensures that $\tilde\mu_{\tilde\eta}$ is concentrated on those jump curves, and
\begin{align*}
\tilde\mu_{\tilde\eta}&=\sum_{j=1}^N (t,x_j(t))_{\sharp}\left\lbrace g_j(t)\, dt\right\rbrace,\\
g_j(t)&=-x_j'(t)\left[
\tilde\eta(v_j(t,x_j(t)))-\tilde\eta(v_{j-1}(t,x_j(t)))\right]
+\tilde q(v_j(t,x_j(t)))-\tilde q(t,v_{j-1}(t))).
\end{align*}
Recalling from the definition of $s_j$ \eqref{eq:defsj} that $x_j'(t)=\sigma(v_{j-1}(t,x_j(t)),v_j(t,x_j(t))) +s_j(t)$, we see that
\begin{align*}
g_j(t)&=E_{\tilde\eta}(v_{j-1}(t,x_j(t)),v_j(t,x_j(t))) \\
&\quad +s_j(t)\left[
\tilde\eta(v_j(t,x_j(t)))-\tilde\eta(v_{j-1}(t,x_j(t)))\right],
\end{align*}
where $E_{\tilde\eta}(u_-,u_+)$ denotes, as in \eqref{eq:E}, the entropy cost of a jump $(u_-,u_+)$ for the entropy $\tilde\eta$. Here, since $v_{j-1}\geq v_j$ we have $E_{\tilde\eta}(v_{j-1},v_j)\leq 0$, and therefore
\begin{align}\label{eq:estimtildemu}
\tilde\mu_{\tilde\eta}&\leq \left(\sup |\tilde\eta'\circ\tilde u| \right)\, \lambda,\\
\lambda&=\sum_{j=1}^N (t,x_j(t))_{\sharp}\left\lbrace \lambda_j(t)\, dt\right\rbrace,
\nonumber
\\
\lambda_j(t)&= \left(v_{j-1}(t,x_j(t))-v_j(t,x_j(t))\right) \abs{s_j(t)} 
\nonumber
\end{align}
Even though $\tilde u$ is not a solution, and its entropy production is not nonpositive, we claim that the proof of Kru\v{z}kov's $L^1$ estimate \cite[Theorem~6.2.3]{dafermos} can be adapted to obtain
\begin{align}\label{eq:kruzkovtilde}
\partial_t\, |u^{ent}-\tilde u| + \partial_x\, Q(u^{ent};\tilde u) \leq \lambda,
\end{align}
where $Q(u;v)=\mathrm{sign}(u-v)(f(u)-f(v))$. We follow \cite[Theorem~6.2.3]{dafermos} and use the variable doubling technique: from \eqref{eq:estimtildemu} and the fact that $u^{ent}$ is an entropy solution we have
\begin{align}\label{eq:doublevar}
(\partial_t +\partial_s)|u^{ent}(s,y)-\tilde u(t,x)|
+(\partial_x+\partial_y)Q(u^{ent}(s,y);\tilde u(t,x))
\leq \lambda(t,x),
\end{align}
and we test this against
\begin{align*}
\Phi_\e(t,x,s,y)=\psi\left(\frac{t+s}{2},\frac{x+y}{2} \right)\rho_\e\left(\frac{t-s}{2}\right)\rho_\e\left(\frac{x-y}{2}\right),
\end{align*}
for any nonnegative test function $\psi(t,x)$, $\rho_\e(t)=\e^{-1}\rho(\e^{1}t)$ with $\rho$ a smooth nongative function with compact support and unit integral, and small enough $\e>0$. In the proof of \cite[Theorem~6.2.3]{dafermos}, it is shown that the left-hand side of \eqref{eq:doublevar} tested against $\Phi_\e$ converges to the left-hand side of \eqref{eq:kruzkovtilde} tested against $\psi$, as $\e\to 0$. To obtain \eqref{eq:kruzkovtilde} we only need to check that the same happens with the right-hand sides. We have
\begin{align*}
&\langle \lambda(t,x),\Phi_\e(t,x,s,y)\rangle \\
&=\int \psi\left(\frac{t+s}{2},\frac{x+y}{2} \right)
\rho_\e\left(\frac{t-s}{2}\right)\rho_\e\left(\frac{x-y}{2}\right)
\,\lambda(dt,dx)\, dsdy\\
&=\int \psi\left(t-\e \hat s,x-\e\hat y \right)\,\lambda(dt,dx)\,
\rho(\hat s)\rho(\hat y) \,d\hat s d\hat y
\xrightarrow[\e\to 0]{}  \langle\lambda,\psi\rangle,
\end{align*}
by dominated convergence, since $\psi$ is bounded and $\lambda$ is a finite measure. This proves \eqref{eq:kruzkovtilde}. (Direct computations using the $BV$ chain rule would also provide a proof.) 

 Testing \eqref{eq:kruzkovtilde} as in the proof of Lemma~\ref{l:var} and discarding nonpositive boundary terms, we deduce that
\begin{align*}
&\int_{x_0(t_2)}^{x_{N+1}(t_2)}\abs{\tilde u(t_2,x)-u^{ent}(t_2,x)}\, dx
-  \int_{x_0(t_1)}^{x_{N+1}(t_2)}\abs{\tilde u(t_1,x)-u^{ent}(t_1,x)}\, dx\\
&\leq \int \mathbf 1_{t_1<t< t_2,x_0(t)<x<x_{N+1}(t)}\lambda(dt,dx)\\
&=\sum_{j=1}^N \int_{t_1}^{t_2} \sqrt{v_{j-1}(t,x_j(t))-v_j(t,x_j(t))}\,\theta_j(t)\, dt,
\end{align*}
where $\theta_j(t)=\sqrt{v_{j-1}(t,x_j(t))-v_j(t,x_j(t))}\, \abs{s_j(t)}$ satisfies \eqref{eq:thetaj} thanks to \eqref{eq:estimsj}.
\end{proof}

Now we turn to the proof of Theorem~\ref{t:stabent} for initial conditions  $\zeta_0$ with a finite number of shocks as in \eqref{eq:zeta0finite}.

\begin{proof}[Proof of Theorem~\ref{t:stabent} for $\zeta_0$ as in \eqref{eq:zeta0finite}]
We let $x_1,\ldots ,x_N\colon [0,T]\to\mathbb R$ be the generalized characteristics of $u$ starting at $x^0_1,\ldots, x^0_N$.
We let 
\begin{align*}
t_*=\sup\left\lbrace t\in [0,T]\colon x_1(\tau) < \cdots < x_N(\tau) \;\forall \tau\in [0,t] \right\rbrace >0,
\end{align*}
so that the curves $x_j$ do not intersect on $[0,t_*)$. If $t_*<T$, some of the curves intersect at $t_*$, and for all $j\in \lbrace 1,\ldots, N\rbrace$ such that 
\begin{align*}
x_{j-1}(t_*) < x_j(t_*) =x_{j+1}(t_*) = \ldots =x_{\ell}(t_*) < x_{\ell+1}(t_*),
\end{align*}
for some $\ell\in\lbrace j+1,\ldots, N\rbrace$
(with the convention that $x_0=-\infty$ and $x_{N+1}=+\infty$),
we modify $x_{j+1},\ldots ,x_\ell$ on $[t_*,T]$ by setting them all equal to $x_j$. In particular after this modification these curves are still generalized characteristics of $u$. We repeat that procedure, at most $N$ times, until we have generalized characteristics $x_1,\ldots,x_n$ starting at $x_0^1,\ldots, x^0_N$, which may intersect, but not cross:
\begin{align*}
x_1\leq x_2\leq \ldots \leq x_N\quad\text{ on }[0,T].
\end{align*}
Then we define $\tilde u$ on $[0,T]\times\R$ by setting
\begin{align*}
\tilde u(t,x) = v_0(t,x)\mathbf 1_{x<x_1(t)} + v_1(t,x) \mathbf 1_{x_1(t)<x<x_2(t)} +\cdots +v_N(t,x)\mathbf 1_{x> x_N(t)},
\end{align*}
where $v_j$ are the entropy solutions of \eqref{eq:scl} with initial data $v^0_j$. Note that some terms of this sum may become zero as $t$ increases and curves merge.

For any $t\in [0,T]$ and $R>0$ we set
\begin{align*}
&x_0(\tau)=-R-S(t-\tau),\quad x_{N+1}(\tau)=R+S(t-\tau),\\
& A_{R,t}=\left\lbrace (\tau,\xi)\colon 0<\tau < t,\; x_0(\tau)< \xi < x_{N+1}(\tau)\right\rbrace.
\end{align*}
As $\tau$ increases from $0$ to $t$, some curves $x_1(\tau),\cdots,x_N(\tau)$ may merge, or exit the interval $[x_0(\tau),x_{N+1}(\tau)]$ (but they can not enter it, as $|x_j'|\leq S$). We partition $[0,T]$ into at most $N$ intervals inside which no merging nor crossing happens. In those subintervals we can apply Proposition~\ref{p:stabfinite}. Concatenating all resulting estimates \eqref{eq:estimL2utilde}, we deduce
\begin{align*}
&\int_{-R}^{R}\abs{u(t_2,x)-\tilde u(t_2,x)}^2 dx
-  \int_{-R-St}^{R+St}\abs{u(0,x)-\tilde u(0,x)}^2  dx
\\
&\leq \Lambda(0,t)= C\frac{M^3}{\alpha^3}\mu_+((0,t)\times (-R-St,R+St)),\nonumber
\end{align*}
which proves \eqref{eq:stabentL2}. Concatenating the estimates \eqref{eq:estimL1utilde} we obtain a function $\theta$ defined on the jump set $J_{\tilde u}$ of $\tilde u$, such that for any entropy solution $u^{ent}$,
\begin{align}\label{eq:stabentL1theta}
&\int_{-R}^{R}\abs{\tilde u(t,x)-u^{ent}(t,x)}\, dx
-  \int_{-R-St}^{R+St}\abs{\tilde u(0,x)-u^{ent}(0,x)}\, dx \\
&
\leq 
\int_{J_{\tilde u}\cap A_{R,t}} |[\tilde u]|^{\frac 12} \theta \, |\nu_x|\,d\mathcal H^1,\nonumber\\
\text{and }
&\frac{\alpha}{CM^2}\int_{J_{\tilde u}\cap A_{R,t}}\theta^2\, |\nu_x|\,d\mathcal H^1 \leq \int_{-R-St}^{R+St}\abs{u(0,x)-\tilde u(0,x)}^2  dx + \Lambda(0,t).\nonumber
\end{align}
Here we use the notations of Remark~\ref{r:Jutilde}, $[\tilde u]$ denotes the jump of $\tilde u$ along the jump set $J_{\tilde u}$ with normal vector $\nu=(\nu_t,\nu_x)$.

Then we estimate the right-hand side of \eqref{eq:stabentL1theta},
\begin{align*}
&\int_{J_{\tilde u}\cap A_{R,t}} |[\tilde u]|^{\frac 12} \theta \, |\nu_x|\,d\mathcal H^1 \\
&
\leq \left(\int_{J_{\tilde u}\cap A_{R,t}}|\tilde u|\,|\nu_x|\,d\mathcal H^1\right)^{\frac 12}
\left(\int_{J_{\tilde u}\cap A_{R,t}}\theta^2\, |\nu_x|\,d\mathcal H^1\right)^{\frac 12}\\
&\leq C \frac{M}{\alpha^{\frac 12}}\left(\int_{J_{\tilde u}\cap A_{R,t}}|\tilde u|\,|\nu_x|\,d\mathcal H^1\right)^{\frac 12}
\sqrt{ \int_{-R-St}^{R+St}\abs{u(0,x)-\tilde u(0,x)}^2  dx + \Lambda(0,t)}.
\end{align*}
Using the nonincreasing property \eqref{eq:vjdiff} of the differences between the functions $v_j$ we can estimate $[\tilde u]$, and obtain, with $X_0=\lbrace j\colon x^0_j\in [-R-St,R+St]\rbrace$,
\begin{align*}
\int_{J_{\tilde u}\cap A_{R,t}}|[\tilde u]|\,|\nu_x|\,d\mathcal H^1 &
\leq \int_0^t \left[\sum_{j\in X_0} \sup_{[x_0(\tau),x_{N+1}(\tau)]} (v_{j-1}(\tau,\cdot)-v_j(\tau,\cdot))\right] \, d\tau
\\
&\leq t \sum_{j\in X_0}d^0_j 
= t\cdot (D\zeta_0)_-([-R-St,R+St]),
\end{align*}
hence
\begin{align*}
\int_{J_{\tilde u}\cap A_{R,t}} |[\tilde u]|^{\frac 12} \theta \, |\nu_x|\,d\mathcal H^1 
&\leq C \alpha^{-\frac 12}\sqrt{(D\zeta_0)_-([-R-St,R+St]) }\sqrt{t}\\
&\quad\quad\cdot
\sqrt{ \int_{-R-St}^{R+St}\abs{u(0,x)-\tilde u(0,x)}^2  dx + \Lambda(0,t)}.
\end{align*}
Combining this with \eqref{eq:stabentL1theta} and choosing $u^{ent}=\zeta$ readily implies \eqref{eq:stabentL1}, and concludes the proof of Theorem~\ref{t:stabent} when the initial condition $\zeta_0$ has a finite number of shocks as in \eqref{eq:zeta0finite}.
\end{proof}

Before turning to the proof of Theorem~\ref{t:stabent} for any initial condition $\zeta_0\in L^\infty\cap BV_{loc}(\R)$, we gather some estimates on the function $\tilde u$ that we just constructed:

\begin{lemma}\label{l:estimutilde}
When $\zeta_0$ has a finite number of shocks as in \eqref{eq:zeta0finite}, for any $0\leq s<t\leq T$ and $R>0$ we have the bounds
\begin{align}
|D_x\tilde u|([s,t]\times [-R,R])&\leq 2(t-s)\left(\norm{\zeta_0}_\infty +(D\zeta_0)_-([-R-St,R+St])\right),\label{eq:dxutilde}\\
|D_t \tilde u|([s,t]\times [-R,R])&\leq  S |D_x\tilde u|([s,t]\times [-R,R]),\label{eq:dtutilde}\\
\int_{-R}^R |\tilde u(t,x)-\tilde u(s,x)|\, dx &\leq 2S \big(\norm{\zeta_0}_\infty+(D\zeta_0)_-([-R-St,R+St])\big)(t-s).\label{eq:utildeLipL1}
\end{align}
\end{lemma}
\begin{proof}[Proof of Lemma~\ref{l:estimutilde}]
To obtain \eqref{eq:dxutilde}, remark that $(D_x\tilde u)_-$ consists only of shocks which are differences between the $v_j$, therefore using the nonincreasing properties \eqref{eq:vjdiff} of such differences and letting  $X_0=\lbrace j\colon x^0_j\in [-R-St,R+St]\rbrace$ we have
\begin{align*}
(D_x\tilde u)_-([s,t]\times [-R,R])&= \int_{J_{\tilde u}\cap ([s,t]\times [-R,R])}|[\tilde u]| \,|\nu_x|\,d\mathcal H^1 \\
&\leq (t-s) \sum_{j\in X_0}d^0_j = (t-s)(D\zeta_0)_-([-R-ST,R+ST]).
\end{align*}
This implies \eqref{eq:dxutilde} since $|D_x\tilde u|=D_x\tilde u + 2 (D_x\tilde u)_-$ and
\begin{align*}
D_x\tilde u ([s,t]\times [-R,R]) \leq 2 (t-s)\norm{\tilde u}_\infty \leq 2(t-s)\norm{\zeta_0}_\infty.
\end{align*}
To obtain \eqref{eq:dtutilde} we simply note that outside $J_{\tilde u}$ the function $\tilde u$ is Lipschitz and satisfies $\partial_t\tilde u=-f'(\tilde u)\partial_x\tilde u$, and for the jump part we take into account that the normal vector satisfies $|\nu_t|\leq S|\nu_x|$. Finally, using that 
\begin{align*}
\int_{-R}^R |\tilde u(t,x)-\tilde u(s,x)|\, dx \leq |D_t\tilde u|([s,t]\times [-R,R]),
\end{align*}
we obtain \eqref{eq:utildeLipL1} as a consequence of \eqref{eq:dxutilde}-\eqref{eq:dtutilde}.
\end{proof}

\begin{remark}\label{r:oleinikutilde}
We also have the Oleinik-type bound $D_x \tilde u(t,\cdot)\leq 1/(\alpha t)$ for all $t\in (0,T]$, since the functions
 $v_j$ satisfy $\partial_x v_j\leq 1/(\alpha t)$ and all other contributions to $D_x\tilde u$ are negative shocks.
\end{remark}

Finally we prove Theorem~\ref{t:stabent} for any initial condition $\zeta_0\in L^\infty\cap BV_{loc}(\R)$.

\begin{proof}[Proof of Theorem~\ref{t:stabent}]
We fix $\zeta_0\in L^\infty\cap BV_{loc}(\R)$ and approximate it with functions $\zeta_\e$ of the form \eqref{eq:zeta0finite} as follows. We have $\zeta_0=\zeta^1 +\zeta^2$ where $\zeta^1$ is nondecreasing and $\zeta^2$ is nonincreasing, and  $\Vert\zeta^j\Vert_\infty\leq \norm{\zeta_0}_\infty$ for $j=1,2$. We fix a smooth compactly supported nonnegative function $\rho$ with unit integral on $\R$ and define $\rho_\e(x)=\e^{-1}\rho(\e^{-1}x)$. We set
\begin{align*}
\zeta^1_\e =\zeta^1 *\rho_\e,
\end{align*}
so that $\zeta^1_{\e}$ is Lipschitz nondecreasing and $\Vert\zeta^1_\e\Vert_\infty\leq \norm{\zeta_0}_\infty$. Taking $\e=1/K$ for some integer $K>0$, we let
\begin{align*}
a=\inf\zeta^2,\quad b=\sup\zeta^2,\quad a_\ell =a+\ell \frac{b-a}{K}\text{ for }\ell\in\lbrace 0,\ldots, K\rbrace
\end{align*}
and for an integer $K>0$ we define
\begin{align*}
\zeta^2_\e(x) =a_0 \mathbf 1_{\zeta^2(x)=a_0} + \sum_{\ell=1}^K a_\ell \mathbf 1_{a_{\ell-1}<\zeta^2(x)\leq a_\ell}.
\end{align*}
Since $\zeta^2$ is nonincreasing, the sets in the above indicator functions are intervals, and we see that 
\begin{align*}
\zeta^0_\e =\zeta^1_\e +\zeta^2_\e,
\end{align*}
is equal almost everywhere to a function of the form \eqref{eq:zeta0finite}. Moreover we have 
\begin{align*}
\zeta^0_\e &\longrightarrow \zeta_0\quad\text{in }L^2_{loc}(\mathbb R)\\
(D\zeta^0_\e)_-([-R,R])&\longrightarrow (D\zeta_0)_-([-R,R])\quad\text{for all }R>0,
\end{align*}
as $\e\to 0$. Applying Theorem~\ref{t:stabent} to the initial condition $\zeta^0_\e$ we obtain functions $\tilde u_\e$ satisfying \eqref{eq:stabentL2}-\eqref{eq:stabentL1}. Thanks to Lemma~\ref{l:estimutilde} and $\Vert\tilde u_\e\Vert_\infty\leq \Vert\zeta_0\Vert_\infty$ we may extract a subsequence of $\tilde u_\e$ such that $\tilde u_\e(t,\cdot)$ converges in $L^2_{loc}$  for every $t\in [0,T]$, and pass to the limit in \eqref{eq:stabentL2}-\eqref{eq:stabentL1}. The limit $\tilde u$ satisfies the bounds of Lemma~\ref{l:estimutilde}.
\end{proof}

\bibliographystyle{acm}
\bibliography{ref}
\end{document}